\documentclass[psamsfonts]{amsart}
\usepackage{enumerate}

\usepackage{graphicx}

\usepackage{amsmath, amssymb, amsfonts}

\usepackage{thmtools, thm-restate}

\newcommand{\size}[1]{\left|#1\right|} 
\newcommand{\pr}{^\prime}
\newcommand{\ts}{\thickspace}
\newcommand{\sym}[1]{\text{Sym}(#1)}
\newcommand{\ket}[1]{|#1\rangle}
\newcommand{\bra}[1]{\langle #1|}
\newcommand{\ketbra}[2]{\ket{#1}\!\bra{#2}}        
\newcommand{\extwk}{\widehat{W}^{(k)}}

\newcommand{\R}{\mathcal{R}}
\newcommand{\Z}{\mathcal{Z}}

\newcommand{\cel}{\text{Cel}}
\newcommand{\matv}{\text{Mat}_V}
\newcommand{\tr}{\text{tr}}

\declaretheorem[name=Definition]{definition}
\declaretheorem[name=Lemma]{lemma}
\declaretheorem[name=Theorem]{theorem}

\begin{document}

\title{Cellular Algebras and Graph Invariants Based on Quantum Walks}
\author{Jamie Smith}
\date{March 1, 2011}
\address{Institute for Quantum Computing and Department of Combinatorics \&
Optimization \\University of Waterloo\\ Waterloo, ON \\Canada\\ N2L 3G1}
\email{ja5smith@iqc.ca}

\begin{abstract} We consider two graph invariants inspired by quantum walks---one in continuous time (\cite{John-King-Gamble:2010lq}) and one in discrete time (\cite{Emms:2005fr,Godsil:2010zr}).  We will associate a matrix algebra called a {\em cellular algebra} with every graph.  We show that, if the cellular algebras of two graphs have a similar structure, then they are not distinguished by either of the proposed invariants.  
\end{abstract}

\maketitle

\section{Introduction}

In this paper, we consider two graph invariants.  Both invariants are inspired by quantum walks---one in continuous time (\cite{John-King-Gamble:2010lq}) and one in discrete time (\cite{Emms:2005fr,Godsil:2010zr}).  Interestingly, the algorithm defined by each invariant is classical; they are inspired by quantum phenomena, but don't require quantum computers.

In Section \ref{ca}, we will explain how to associate a matrix algebra with any graph.  These algebras, called {\em cellular algebras} encode useful structural information about the underlying graph.  In Sections \ref{emms} and \ref{gamble}, we will show that the invariant of Emms et al. as well as that of Gamble et al. can be described in terms of cellular algebras.  In fact, we will show that, if the cellular algebras arising from two non-isomorphic graphs have a similar structure, then these invariants will fail to distinguish them.

Emms et al. (\cite{Emms:2005fr}) consider the discrete time quantum walk on the directed arcs of the graph $G$.  The walk is governed by the transition matrix $U$, defined as follows
$$U_{wx,uv}=\begin{cases}\frac{2}{\deg(v)}&v=w,u\neq x\\\frac{2}{\deg(v)}-1&v=w,u=x\\0&o.w.\end{cases}$$
Thy then define the {\em positive support} of a matrix:
$$S^+(M)_{uv}=\begin{cases}1&M_{uv}>0\\0&o.w.\end{cases}$$
In the papers of Godsil and Guo, and of Emms et al., they consider the matrix
$$S^+(U^3)$$
In particular, they consider the spectrum of this matrix.  Clearly the spectrum is a graph invariant.  Emms et al. conjecture that it is strong enough to distinguish any pair of strongly regular graphs.  We derive a sufficient condition for two graphs to be indistinguishable by this invariant.  This result, proved in section \ref{emms} is stated as follows:

\begin{restatable}{theorem}{emmscospec} Let $G$ and $G\pr$ be 2-equivalent graphs with transition matrices $U$ and $\overline U$.  Then,
$$\text{spec}(S^+(U^3))=\text{spec}(S^+(\overline U^3))$$
\end{restatable}

In \cite{John-King-Gamble:2010lq}, Gamble et al. consider the application of a multi-paritcle quantum walks to the graph isomorphism problem.  Depending on the particular nature of the particles and their interaction, we assign a Hamiltonian $H$ to a graph $G$.  The unitary operator describing the system is then given by 
$$U=U(t)=e^{-itH}$$

The Green's functions associated with this graph at some time $t$ are the values $$\mathcal{G}(i,j)=\bra{j}U\ket{i}$$ where $i$ and $j$ run over an appropriate basis (in the case of \cite{John-King-Gamble:2010lq}, this is the two-particle basis).  The invariant of Gamble et al. is the set of Green's functions.

Gamble et al. focus on the case of interacting pairs of Bosons.  They test this invariant on all tabulated pairs of non-isomorphic strongly regular graphs up to 64 vertices, finding that it does indeed distinguish them.  They ask whether, for a high enough value of $k$, the $k$-Boson walk could distinguish all non-isomorphic graphs.  For any $k$, we describe a sufficient condition for a pair of graphs to be indistinguishable by the $k$-Boson walk.  This result, which will be described in more depth in Section \ref{gamble}, is summed up as follows:

\begin{restatable}{theorem}{kequivkbos}\label{kequivkbos}If $G$ and $G\pr$ are $k$-equivalent graphs, then they are not distinguished by the interacting $k$-Boson quantum walk.\end{restatable}

It is known that pairs of $k$-equivalent graphs exist for all $k$.  Such constructions are outlined in \cite{Ponomarenko:2009vn}.  These constructions are not, in general strongly regular.  The existence of strongly regular pairs of $k$-equivalent graphs remains an open question, even when restricted to $k=2$.  While finding such a pair would be a significant result, proving that no strongly regular $k$-equivalent pair exists would be a significant step toward an algorithm for the isomorphism problem in strongly regular graphs.

\section{Cellular Algebras}\label{ca}


Cellular algebras are a generalization of coherent configurations, which were developed by Weisfeiler and Lehman (\cite{B.Yu.-Weisfeiler:1968mz}) and Higman (\cite{Higman:1970ve}) as an approach to the graph isomorphism problem.  We will see how these algebras are generated from graphs;  these generated cellular algebras capture structural information about the underlying graph that we will use to prove the main theorem of this paper.

Let $V$ be a finite vertex set.  Then, $\matv$ is the algebra of all complex-valued matrices indexed by $V$.  We define a cellular algebra as follows:

\begin{definition} Let $W$ be a subalgebra of $\matv$.  Then, $W$ is a cellular algebra if the following hold:
\begin{enumerate}[(i)]
\item $W$ is closed under Hadamard multiplication $\circ$.
\item $W$ is closed under complex conjugation $^\dagger$.
\item $W$ contains the identity $I$ and the all-ones matrix $J$.
\end{enumerate}
\end{definition}

Then, if $W$ is a cellular algebra,
\begin{enumerate}[(i)]
\item $W$ has a unique basis of $0-1$ matrices $\R$, and $\sum_{R\in\R}R=J$.
\item There is a subset $C\subseteq\R$ such that $\sum_{R\in\R}R=I$.
\item If $R\in\R$, then $R^\dagger\in\R$
\end{enumerate}

We call $\R$ the set of {\bf basis relations}\footnote{The term {\em basis relation} comes from an alternative  definition, in which we consider binary relations on $V$, rather than the equivalent $0-1$ matrices.} of $W$.  We will also use $\R^*$ to denote the set of all sums of elements of $\R$; this is the set of {\bf relations} of $W$.  A set of vertices $U\subseteq V$ is called a {\bf cell} of $W$ if $I_U$, the identity on $U$, is a basis relation of $W$.  The set of cells of $W$ is denoted by $\cel(W)$.

\subsection{The Cellular Closure: Cellular Algebras from Graphs}\label{secsrg}

The cellular algebra $W=[M_1,...,M_\ell]$ is the smallest cellular algebra containing $\{M_1,...,M_\ell\}$, a set of $n\times n$ matrices.  We say that $W$ is {\bf generated} by $\{M_1,...,M_\ell\}$.  If $G=(V,E)$ is a graph with adjacency matrix $A$, then we say that $[A]$ is the {\bf cellular closure} of $G$.  This will sometimes be denoted by $[G]$.  If a cellular algebra $W$ contains $[G]$, then we will say that $W$ contains the graph $G$.  Strongly regular graphs have the simplest cellular closures:

\begin{definition}\label{srg}
A {\bf strongly regular graph} $G$ is associated with a set of parameters $(n,k,\lambda,\mu)$ such that:
\begin{enumerate}[(i)]
\item $G$ has $n$ vertices.
\item Each vertex has degree $k$.
\item Each pair of adjacent vertices share $\lambda$ common neighbours.
\item Each pair of non-adjacent vertices share $\mu$ common neighbours.
\end{enumerate}
\end{definition}

If a strongly regular graph $G$ has adjacency matrix $A$, it is easily verified that $\{I, A, (J-I-A)\}$ form the basis for the cellular algebra $[A]$.

The following lemma shows that the cells of $[G]$ distinguish vertices based on (among other characteristics) their degree:
\begin{restatable}{lemma}{degreecell}\label{degreecell}  If $\deg{u}\neq\deg{v}$, then $u$ and $v$ are in different cells of $[G]$.\end{restatable}

In other words, if $I_d$ is the identity on all vertices of degree $d$, and $J_d$ is the all-ones matrix on vertices of degree $d$, then $I_d, J_d\in [G]$.

While most generated cellular algebras are not as straightforward as this, the Weisfeiler-Lehman algorithm (see \cite{B.Yu.-Weisfeiler:1968mz}, \cite{Evdokimov:1999qf}) calculates the cellular closure of a set of matrices in polynomial time.

\subsection{Weak and Strong Isomorphisms}

We will define two notions of isomorphisms between cellular algebras--- one of a combinatorial nature (strong), and the other of an algebraic nature (weak).  Let $W$ and $W\pr$ be cellular algebras with vertex sets $V$ and $V\pr$ and basis relations $\R$ and $\R\pr$, respectively. 

\begin{definition} A {\bf weak isomorphism} is a bijection $\phi\colon W\to W\pr$ that preserves addition, matrix multiplication, Hadamard multiplication and complex conjugation.
\end{definition}
Two immediate consequences of this definition are that $\phi(I)=I$, and $\phi$ is a bijection from the basis relations $\R$ to $\R\pr$.  We also note that $\phi$ induces a bijection between cells, $\phi\pr\colon \cel(W)\to\cel(W\pr)$.  Weak isomorphisms respect algebraic structure, but do not take into account the vertices underlying the cellular algebra.  We now define a stronger notion of isomorphism.

\begin{definition} A {\bf strong isomorphism} is a bijection $\psi\colon V\to V\pr$ such that, for each $R\in W$, there is a unique $R\pr\in W\pr$ such that
$$\forall u,v\in V,\: R(u,v)=R\pr(\psi(u),\psi(v)).$$
\end{definition}

The following lemmas will be very useful when considering the graph invariants of sections \ref{emms} and \ref{gamble}.  Their proofs can be found in Appendix A.

\begin{restatable}{lemma}{celltocell}\label{celltocell}
Take $X\in\cel(W)$.  Then $\size{X}=\size{\phi\pr(X)}$.\end{restatable}

\begin{restatable}{lemma}{treq}\label{treq}
For all $R\in W$, $\tr(R)=\tr(\phi(R))$.
\end{restatable}

\begin{restatable}{lemma}{cospectral}\label{cospectral} Let $\phi:W\to W\pr$ be a weak isomorphism and take any $A\in W$ with $\phi(A)=A\pr$.  Then, $A$ and $A\pr$ are cospectral.\end{restatable}

\subsection{Cellular Algebra Extensions}
The $k$-extension of a cellular algebra $W$ is a larger algebra that contains additional structural information about $W$.  Before constructing the $k$-extension, we first need to define the centralizer algebra.
\begin{definition}
Let $G$ be a group acting on the set $S$.  The {\bf centralizer algebra} is defined as follows:
$$\Z(G,S)=\{A\in\text{Mat}_S\colon \forall g\in G,\;A^g=A\}$$
\end{definition}
In the next definition, we use the centralizer algebra $\Z(\sym{V},V^k)$.  In this case, $\sym{V}$ acts entrywise on $V^k$.
\begin{definition}
The {\bf $k$-extension} $\extwk$ of a cellular algebra $W$ is the smallest cellular algebra containing $W^k$ and $\Z(\sym{V},V^k)$:
$$\extwk=[W^k,\Z(\sym{V},V^k)]$$
\end{definition}
The following lemma is reproduced from \cite{Ponomarenko:2009vn}, and was originally proven in \cite{Evdokimov:1999qf}:
\begin{lemma}\label{cylindric1} Let $S=\{R_{i,j}\colon 1\leq i,j\leq k\}\subseteq \R^*$ be a set of relations.  Define the cylindric relation $\text{Cyl}_S$ such that, given $\overline{x},\overline{y}\in V^k$,
$$\text{Cyl}_S(\overline{x},\overline{y})=\prod_{i,j}R_{i,j}(x_i,y_j).$$
Then, $\text{Cyl}_S\in\extwk$.
\end{lemma}

Taking a set of relations $T=\{L_{i,j}\colon 1\leq i<j\leq k\}\subseteq\mathcal{R}^*$, we say that $\overline{x}\in V^k$ is of type $T$ if $(x_i,x_j)\in L_{i,j}$ for all $1\leq i<j\leq k$.  As an immediate consequence of the theorem above, we get the following lemma:
\begin{restatable}{lemma}{typet} \label{typet}Let $T=\{L_{i,j}\colon 1\leq i,j\leq k\}$ and let $I_T$ be the identity relation restricted to elements of $V^k$ of type $T$.  Then,
$$I_T\in \extwk$$
\end{restatable}

A proof can be found in Appendix A.

\section{The $k$-Equivalence of Graphs} \label{kequiv}

We would like to use these ideas of weak isomorphism and $k$-extension to draw meaningful connections between graphs.

\begin{definition} Let $G$ and $G\pr$ be graphs with adjacency matrices $A$ and $A\pr$.  Then, $G$ and $G\pr$ are {\bf equivalent} if there is a weak isomorphism $\phi\colon [G]\to[G\pr]$ such that $\phi(A)=A\pr$.  We say that $\phi$ is an equivalence from $G$ to $G\pr$.\end{definition}

Since $[G]$ is the smallest cellular algebra containing $A$, the weak isomorphism $\phi$ is in fact uniquely determined.  

\begin{lemma} Let $\phi$ is an equivalence from $G$ to $G\pr$ and $I_d$ and $I_d\pr$ be the identity operators on vertices of degree $d$.  Then,
$$\phi(I_d)=I_d\pr$$
\end{lemma}
\begin{proof}  Letting $A$ and $A\pr$ be the adjacency matrices of $G$ and $G\pr$, we can write 
$$A^2=\sum_d{dI_d}+\sum_{R\text{ non-diagonal}}a_RR$$
and 
$$(A\pr)^2=\sum_d{dI_d\pr}+\sum_{R\pr\text{ non-diagonal}}a_R\pr R\pr$$
Then, since $\phi(A^2)=\phi((A\pr)^2)$ and $\phi$ takes diagonal matrices to diagonal matrices we must have $\phi(I_d)=I_d\pr$.
\end{proof}

We now broaden these definitions to take into account $k$-extensions:

\begin{definition} Let $G$ and $G\pr$ be graphs with adjacency matrices $A$ and $A\pr$.  Then $\phi$ is a $k$-equivalence if
\begin{enumerate}[(i)]
\item It is an equivalence from $G$ to $G\pr$.
\item There exists a weak isomorphism $\widehat\phi\colon \widehat{[G]}^{(k)}\to \widehat{[G\pr]}^{(k)}$ such that
$$\widehat\phi\mid_{[G]^k}=\phi^k\qquad and \qquad\widehat\phi\mid_{\Z(\sym{V},V^k)}=I$$
\end{enumerate}
\end{definition}


Clearly, 1-equivalence corresponds with our existing definition of graph equivalence.  A version of the following lemma appears in \cite{Ponomarenko:2009vn} and gives a useful description of the action of a $k$-equivalence.

\begin{lemma}\label{cylindric2} Let $W$ and $W\pr$ be cellular algebras containing $G$ and $G\pr$.  Let $\phi:W\to W\pr$ be a $k$-equivalence from $G$ to $G\pr$.  Let $S\subseteq \R^*$ be a set of relations of $W$.  Then,
$$\widehat\phi(\text{Cyl}_S)=\text{Cyl}_{\phi(S)}$$
where $\phi(S)=\{\phi(R_{i,j})\colon R_{i,j}\in S\}$.
\end{lemma}

\section{A Discrete Time Quantum Walk Invariant}\label{emms}

Given an undirected graph $G$, we can define a classical random walk using the stochastic matrix $B$ where
$$B_{uv}=\begin{cases}\frac{1}{\deg(v)}&uv\in E(G)\\0&o.w.\end{cases}$$

To define a quantum analogue, we construct a unitary transition matrix $U$.  First, we will define the digraph $D=D(G)$ by replacing each edge $uv$ of $G$ by the directed arcs $uv$ and $vu$.  The quantum walk takes place on the arcs of $D$.  Its entries of $U=U(G)$ are defined as follows:
$$U_{wx,uv}=\begin{cases}\frac{2}{\deg(v)}&v=w,u\neq x\\\frac{2}{\deg(v)}-1&v=w,u=x\\0&o.w.\end{cases}$$

We now define the {\em positive support} of a matrix as follows:
$$S^+(M)_{uv}=\begin{cases}1&M_{uv}>0\\0&o.w.\end{cases}$$
In the papers of Godsil and Guo, and of Emms et al., they consider the matrix
$$S^+(U^3)$$
In particular, they consider the spectrum of this matrix.  Clearly the spectrum is a graph invariant.  Emms et al. conjecture that it is strong enough to distinguish any pair of strongly regular graphs.

In this section, we will show that the transition matrix $U$ is contained in the 2-extension of a graph.  This will lead us to the main theorem of the section:

\emmscospec*

\subsection{The matrix $S^+(U^k)$ in the cellular algebra $\widehat{W}^{(2)}$}\label{unit}

First, we would like to show that $S^+(U^k)$ is contained in the 2-extension $\widehat{W}^{(2)}$. 
This tells us that the 2-extension captures the structure encoded in $S^+(U^k)$.  Moreover, we will show that any 2-equivalence must respect the operator $S^+(U^k)$.  This will lead us to our main theorem---that a pair of 2-equivalent graphs are not distinguished by the spectrum of $S^+(U^k)$ for any $k$.

\begin{lemma} Let $\phi$ be a 2-equivalence from $G$ to $G\pr$.  Let $W$ and $W\pr$ be cellular algebras containing $G$ and $G\pr$ respectively, and $U=U(G)$ and $\overline U=U(G\pr)$ be their corresponding transition matrices.  Then, for all $k\in\mathbb{Z}^+$,
\begin{enumerate}[(i)]
\item $U^k\in\widehat{W}^{(2)}$ and $\overline U^k\in\widehat{W\pr}^{(2)}$.
\item $S^+(U^k)$ and $S^+(\overline U^k)\in\widehat{W\pr}^{(2)}$.
\item $\widehat{\phi}(S^+(U^k)=S^+(\overline U^k)$
\end{enumerate}
\end{lemma}
\begin{proof}
Define the sets of relations 
\begin{flalign*}
R_{1,1}=&J\\
R_{1,2}=&I\\
R_{2,1}=&J-I\\
R_{2,2}=&J
\end{flalign*}
and set $S=\{R_{1,1},R_{1,2},R_{2,1},R_{2,2}\}$.  Further define
\begin{flalign*}
R\pr_{1,1}=&J\\
R\pr_{1,2}=&I\\
R\pr_{2,1}=&I\\
R\pr_{2,2}=&J
\end{flalign*}

and set $S\pr=\{R\pr_{1,1},R\pr_{1,2},R\pr_{2,1},R\pr_{2,2}\}$.  Then, by Lemma \ref{cylindric1}, $Cyl_{S}, Cyl_{S\pr}\in\extwk$.  Now, in the notation of Lemma \ref{typet}, set ${}^dL_{1,2}=I_dA$.  Then, if $T_d=\{{}^dL_{1,2}\}$, then $I_{T_d}$ is the identity on pairs of vertices that are adjacent in $G$, where the first vertex is of degree $d$.  Also, note that $\phi(I_dA)=I_d\pr A\pr$, so $\widehat\phi$ maps $I_{T_d}$ to $I_{T_d\pr}\in\widehat{W\pr}^{(2)}$---the identity relation on edges whose first vertex has degree $d$.  We can now express the transition matrix $U=U(G)$ as 
$$U(G)=\sum_dI_{T_d}\Bigl(\frac{2}{d}Cyl_S+(\frac{2}{d}-1)Cyl_{S\pr}\Bigr)$$

Therefore, $U\in\widehat{W}^{(2)}$ and therefore $U^k\in\widehat{W}^{(2)}$ for all $k$.  The proof is identical for $\overline U$ and $\widehat{W\pr}^{(2)}$

As a result of (i), we can write any $U^k$ as 

$$U^k=\sum_{R\in\mathcal R}a_RR$$

Now, let $\mathcal R_+=\{R\in \mathcal R:\ts a_R>0\}$.  Then,
$$S^+(U^k)=\sum_{R\in\mathcal R_+}R$$
So, $S^+(U^k)\in \widehat{W}^{(2)}$ for all $k$.  The proof for $S^+(\overline U^k)$ is identical.

Note that Applying Lemma \ref{cylindric2}, we see that 
\begin{flalign*}\widehat{\phi}(U)=&\widehat{\phi}\left(\sum_dI_{T_d}\Bigl(\frac{2}{d}Cyl_S+(\frac{2}{d}-1)Cyl_{S\pr}\Bigr)\right)\\
=&\sum_dI_{T_d\pr}\Bigl(\frac{2}{d}Cyl_{\phi(S)}+(\frac{2}{d}-1)Cyl_{\phi(S\pr)}\Bigr)
=\overline{U}
\end{flalign*}

Then, letting $\mathcal R$ be the basis of $\widehat{W}^{(2)}$, we can write
$$U=\sum_{R\in\mathcal R}a_RR$$
and
$$\overline U=\widehat{\phi}(U)=\sum_{R\in\mathcal R}a_R\widehat\phi(R)$$

Let $\mathcal{R}_+=\{R\in\mathcal R:\ts a_R>0\}$.  Then,
\begin{flalign*}
S^+(\overline U)=&\sum_{R\in\mathcal R_+}\widehat\phi(R)\\
=&\widehat\phi\left(\sum_{R\in\mathcal R_+}R\right)\\
=&\widehat\phi(S^+(U))
\end{flalign*}

proving (iii).
\end{proof}

This leads us directly to the main result of this section:

\emmscospec*
%
%
%

\section{A Continuous Time Quantum Walk Invariant}\label{gamble}

In a 2010 paper, Gamble, Friesen, Zhou, Joynt and Coppersmith \cite{John-King-Gamble:2010lq} consider the application of a multi-paritcle quantum walks to the graph isomorphism problem.   They attempt exploit the interaction between the particles to distinguish graphs that can not be distinguished by classical random walks or single particle quantum walks.  In particular, they examine (a) two non-interacting Bosons; (b) two non-interacting Fermions; and (c) two interacting Bosons.  In this section, we focus on the case of interacting Bosons.

A graph $G=(V,E)$ is associated with an Hamiltonian $H$; this Hamiltonian will depend on the dynamics and number of particles chosen.  The corresponding unitary operator is then given by
$$U(t)=e^{-itH}$$

The {\em Green's functions} associated with this graph are the values $$\mathcal{G}_t(i,j)=\bra{j}U(t)\ket{i}$$ where $i$ and $j$ run over an appropriate basis (in the case of \cite{John-King-Gamble:2010lq}, this is the two-particle basis).  We say that two graphs are distinguishable if their sets of Green's functions differ.

Gamble et al. prove that non-interacting Bosons and Fermions fail to distinguish pairs of non-isomorphic strongly regular graphs with the same parameters.  In order to evaluate the effectiveness of interacting Bosons, they consider all tabulated pairs of non-isomorphic strongly regular graphs with up to 64 vertices.  They find that the interacting 2-Boson walk does indeed distinguish all the pairs of strongly regular graphs they considered.  They go on to ask if, for a high enough value of $k$, an interacting $k$-Boson quantum walk could distinguish all pairs of non-isomorphic graphs.  In this section, we show that pairs of $k$-equivalent graphs are not distinguished by $k$-Boson quantum walks, answering their question in the negative.

\subsection{Unitary Evolution and Cellular Algebras}\label{unit}

Let $H$ be a Hamiltonian and $W=(V, \R)$ a cellular algebra containing $H$.  Then, unitary $U$ corresponding to $H$ can be written as 

$$U=e^{-itH}=\sum_{j=0}^\infty\frac{(-itH)^j}{j!}$$

We will now take advantage of the fact that $U$ lies within $W$ to express the values of the Green's function in a convenient way.  Since $H\in W$, we can write each $H^j/j!$ as a linear combination of basis relations:

$$\frac{H^j}{j!}=\sum_{R\in\R}p_R(j)\cdot R$$

This gives us

$$U=e^{-itH}=\sum_{R\in\R}\left[R\cdot\sum_{j=0}^\infty p_R(j)\cdot(-it)^j\right]$$

Since each $R\in\R$ is a 0-1 matrix and the Green's functions are simply the entries of the matrix $U$, the values of the Green function can be expressed as follows:
$$x_R(t)=\sum_{j=0}^\infty p_R(j)\cdot(-it)^j$$
each with multiplicity 
$$m_R=\text{sum}(R)=\text{tr}(RR^T)$$

The following lemma makes use of these expression for Green's functions.  It shows that these expressions are preserved by a weak isomorphism.  This forms the basis for our proof that $k$-equivalent graphs are indistinguishable.

\begin{lemma}\label{indistinguishable}Let $H$ and $H\pr$ be two Hamiltonians contained in cellular algebras $W$ and $W\pr$ respectively.  Furthermore, let $\phi:W\to W\pr$ be a weak isomorphism such that $\phi(H)=H\pr$. Then, the Green's functions for the corresponding unitaries $U(t)$ and $U\pr(t)$ take on the same values with the same multiplicities.\end{lemma}
\begin{proof}    Define $p\pr_R(j)$, $x\pr_R(t)$ and $m\pr_R$ as above.  Since $\phi$ is a weak isomorphism,
$$p_R(j)=p_{\phi(R)}\pr(j)$$
and therefore
$$x_R(t)=x\pr_{\phi(R)}(t)$$
Corollary \ref{treq} tells us that
$$m_R=\text{tr}(RR^T)=\text{tr}(\phi(RR^T))=m\pr_{\phi(R)}$$

Therefore, the Green functions take on the same values with the same multiplicities.
\end{proof}

\subsection{Interacting 2-Boson Walks and 2-Equivalence}
\label{2boson}

We will first consider the case of two interacting Bosons.  We will extend this to to $k$ particles in the following section, but we include the 2-Boson case separately as it directly addresses \cite{John-King-Gamble:2010lq}.  In \cite{John-King-Gamble:2010lq}, the Hamiltonian for a two-Boson quantum walk is given by:

$$H_{2B}=-\frac{1}{2}(I+S)A^{\oplus2}+UR$$

where $U$ is a constant energy cost and

\begin{flalign*}
S=&\sum_{i,j}\ketbra{ij}{ji},\\
R=&\sum_i\ketbra{ii}{ii},\\
A^{\oplus n}=&(\overbrace{A\otimes I\otimes...\otimes I}^n)+...+(I\otimes I\otimes...\otimes A)
\end{flalign*}

The following lemma is a direct consequence of the definition of 2-extension, as well as Lemma \ref{cylindric2} regarding cylindric relations:
\begin{lemma} \label{HtoH} If $\phi$ is a 2-equivalence from $G$ to $G\pr$ with corresponding 2-Boson Hamiltonians $H_{2B}$ and $H\pr_{2B}$, then
\begin{enumerate}[(i)]
\item $\widehat\phi(S)=S$
\item $\widehat\phi(A^{\oplus2})=(A\pr)^{\oplus2}$
\item$\widehat\phi(R)=R$
\end{enumerate}
and therefore $\widehat{\phi}(H_{2B})=H\pr_{2B}$.

\end{lemma}
\begin{proof} See the proof of Lemma \ref{HtoHk} for a more general proof.\end{proof}

Combining Lemma \ref{indistinguishable} and Lemma \ref{HtoH}, we arrive at the following:

\begin{theorem} \label{2bos} If $G$ and $G\pr$ are 2-equivalent graphs, then they are not distinguished by the interacting 2-Boson walk.\end{theorem}

\section{Interacting $k$-Boson Walks and $k$-Equivalence}
\label{kboson}

We will now consider the $k$-boson case.  We will consider each term of the Hamiltonian in turn. First, we replace the term $(I+S)$ from the 2-particle Hamiltonian with $\sum_{\sym{k}}S$.  That is, we will work within the subspace invariant under any permutation of the $k$ particles.  Next we replace the term $A^{\oplus2}$ with $A^{\oplus k}$.  

Finally, we consider the interaction term.  We would like the energy contribution from each site to be a function of the number of particles at that site.  Let $\overline{x}\in V^k$ be a basis state of the $k$-Boson system.  Then, define $v_x$ as the number of particles at vertex $v$ in state $x$ and $V_x=\{v_x:\;v\in V\}$.  Then, we can partition $V^k$ into equivalence classes $X_1,...,X_\ell$ such that $x$ and $y$ are in the same class if and only if $V_x=V_y$.  Let $R_i=\sum_{x\in X_i}\ketbra{x}{x}$.  Finally, to each of the $X_i$, we assign an energy penalty $U_i$.  This gives us the interaction term, $\sum_{i=1}^\ell U_iR_i$

Putting these together, we arrive at our $k$-Boson Hamiltonian:

$$H_{kB}=-\frac{1}{k!}\left(\sum_{\sym{k}}S\right)A^{\oplus k}+\sum_{i=1}^{\ell}U_i R_i$$

Note that this expression allows for a good deal of flexibility in the nature of the interaction between particles.  In particular, it includes the non-interacting case, as well as the Bose-Hubbard model, in which the contribution from each site is proportional to the square of the number of particles at that site.

We are now ready to state the following lemma, which is proved in Appendix A:

\begin{restatable}{lemma}{HtoHk} \label{HtoHk} If $\phi$ is a k-equivalence from $G$ to $G\pr$ with corresponding $k$-Boson Hamiltonians $H_{kB}$ and $H\pr_{kB}$, then
\begin{enumerate}[(i)]
\item $\forall S\in \sym{k},\;\widehat\phi(S)=S$
\item $\widehat\phi(A^{\oplus k})=(A\pr)^{\oplus k}$
\item $\forall i,\;\widehat\phi(R_i)=R_i$
\end{enumerate}
and therefore $\widehat{\phi}(H_{kB})=H_{kB}\pr$.
\end{restatable}

Applying Lemma \ref{indistinguishable} gives us the following generalization of Theorem \ref{2bos}:
\kequivkbos*

\section{Conclusion}

We have shown that the quantum walk processes described in \cite{Emms:2005fr} and \cite{John-King-Gamble:2010lq} on a graph $G$ are closely related to the structure of the cellular closure of $G$ and its extensions.  In particular, we show that 2-equivalent graphs are indistinguishable by the method of Emms et al. and $k$-equivalent graphs are indistinguishable by the $k$-Boson method of Gamble et al.

As mentioned above, we can construct pairs of $k$-equivalent graphs for any $k$.  These constructions are outlined in by Ponomarenko and Barghi in \cite{Ponomarenko:2009vn}.  They are based on the well-known CFI graphs constructed by Cai, F\"urer and Immerman \cite{Cai:1992p299}.

These results give rise to several natural questions.  The graph constructions of \cite{Ponomarenko:2009vn} are generally not strongly regular.  Can we find strongly regular examples?  Proving that no such examples exist would be an even more interesting result.  In addition, proving the converse of the results in this paper would be very significant.  That is, can we show that, if a pair of graphs are not $k$-equivalent, the $k$-Boson walks on these graphs display different behaviour?  If they are not 2-equivalent, can they be distinguished by the invariant of Emms et al.?  Answering these questions seems to be very difficult, but could have significant implications for the complexity of the graph isomorphism problem.

\appendix
\section{Proofs of Selected Lemmas}

\degreecell*
\begin{proof} Assume that $u$ and $v$ are in the same cell $C$.  Then there is a basis element $I_C$ in $[G]$.  In this case, we can write
$$A^2=a_CI_C+\sum_{R\neq I_C}a_rR$$
This implies that $A^2_{u,u}=A^2_{v,v}=a_C$.  However, $A^2_{u,u}=\deg{u}$, so this cannot hold if $\deg{u}\neq\deg{v}$.
\end{proof}

\celltocell*
\begin{proof} Let $\R_X$ be the set basis relations of $W$ restricted to the vertex set $X$.  Define $\R\pr_{\phi\pr(X)}$ analogously.  We will first show that $J_X$, the all-ones matrix on the set $X$, is mapped to $J_{\phi\pr(X)}$, the all-ones matrix on the corresponding cell in $W\pr$.  Since, $\phi(I_X)=I_{\phi\pr(X)}$ and for any $R\in\R_X$,
$$I_X\cdot R\cdot I_X=R$$
it follows that 
$$I_{\phi\pr(X)}\cdot\phi(R)\cdot I_{\phi\pr(X)}=\phi(R).$$

Therefore, if $R\in\R_X$, then $\phi(R)\in\R\pr_{\phi\pr(X)}$.  Moreover,

$$\phi(J_X)=\sum_{R\in\R_X}\phi(R)=\sum_{R\in\R\pr_{\phi\pr(X)}}R=J_{\phi\pr(X)}$$

Now, comparing
$$J_X\cdot J_X=\size{X}\cdot J_X$$
and
$$J_{\phi\pr(X)}\cdot J_{\phi\pr(X)}=\size{\phi\pr(X)}\cdot J_{\phi\pr(X)}$$
gives us $\size{X}=\size{\phi\pr(X)}$.

\end{proof}
\treq*
\begin{proof} Let $C$ and $C\pr$ be the basis relations of $W$ and $W\pr$ that sum to the identity.  Then, for each $R\in W$ and $I_X\in C$,
$$R\circ I_X=q_R(X)\cdot I_X$$
for some $q_R(X)\in \mathbb{C}$ and
$$\tr(R)=\sum_{X\in\cel(W)}q_R(X)\cdot\size{X}$$
Since $\phi$ is a weak isomorphism, and applying Lemma \ref{celltocell},
\begin{flalign*}
\tr(\phi(R))=&\sum_{X\in\cel(W)}q_R(X)\cdot\size{\phi\pr(X)}\\
=&\tr(R)
\end{flalign*}
\end{proof}

\cospectral*
\begin{proof}
First, we see that if $f(x)$ is the minimal polynomial of $A$, then
$$\phi(f(A))=f(A\pr)=0$$
The converse is also true: If $g(x)$ is the minimal polynomial of $A\pr$, then 
$$\phi^{-1}(g(A\pr))=g(A)=0$$
So, since both polynomials must be minimal, we have $f=g$ and $A$ has the same eigenvalues as $A\pr$.  Now, we need only verify that these eigenvalues occur with the same multiplicity.

The following argument is presented in \cite{Ponomarenko:2009vn}  Let $\lambda_1,...,\lambda_\ell$ be the distinct eigenvalues of $A$ and $A\pr$ with multiplicities $m_1,...,m_\ell$ and $\overline{m}_1,...,\overline{m}_\ell$ respectively.  Since $\tr(A)=\tr(A\pr)$, we have

$$\sum_j(m_j-\overline{m}_j)\lambda_j^k=0$$

for all $k$.  Taking the values $k=0,...,\size{V}-1$, we get $\size{V}$ equations with unknowns $(m_j-\overline m_j)$.  The resulting matrix is the well known Vandermonde matrix with determinant 
$$\prod_{i\neq j}(\lambda_i-\lambda_j)\neq0$$
Therefore, we have $m_j=\overline m_j$ for all $j$, so $A$ and $A\pr$ are cospectral.
\end{proof}

\typet*
\begin{proof} Setting 
$$R_{i,j}=\begin{cases}L_{i,j}&i<j\\L_{i,j}^T&j>i\\I&i=j\end{cases}$$
and $S=\{R_{i,j}:1\leq i,j\leq k\}$, we see that $I_T=Cyl_S$, so by Lemma \ref{cylindric1}, $I_T\in\extwk$ and the set of points of type $T$ is a union of cells of $\extwk$.
\end{proof}

%

\HtoHk*
\begin{proof} First, we note that any $S\in\sym{k}$ is simply a cylindric relation, with $R_{i,j}\in\{I,J\}$ for all $(i,j)$.  Since $\widehat\phi(I)=I$ and $\widehat\phi(J)=J$, we can apply Lemma \ref{cylindric2} to prove (i).

Similarly, each $R_i$ is a cylindric relation with $R_{i,j}\in \{I,(J-I)\}$, so the same reasoning implies that (iii).

Finally, the definition of $k$-equivalence requires that $\widehat\phi\mid_{W^k}=\phi^k$ where $\phi:W\to W\pr$ is a weak isomorphism such that $\phi(A)=A\pr$.  Therefore, 
\begin{flalign*}
\widehat\phi(A\otimes....\otimes I)=&(A\pr\otimes....\otimes I),\\
\widehat\phi(I\otimes A\otimes....\otimes I)=&(I\otimes A\pr\otimes....\otimes I),\\
&\vdots\\
\widehat\phi(I\otimes...\otimes A)=&(I\otimes...\otimes A\pr)
\end{flalign*}
and $\widehat\phi(A^{\oplus k})=(A\pr)^{\oplus k}$, proving (ii).

Combining (i), (ii) and (iii) gives us $\widehat{\phi}(H_{kB})=H_{kB}\pr$.
\end{proof}

\bibliographystyle{plain}
\bibliography{CellularAlgebraResults}

\end{document}